\documentclass[12pt]{amsart}
\usepackage{latexsym,amssymb}
\usepackage[T1]{fontenc}
\usepackage{amscd,verbatim}
\usepackage[colorlinks,linkcolor=blue,citecolor=blue,urlcolor=red]{hyperref}

\newcommand{\sD}{\mathcal{D}}
\newcommand{\sF}{\mathcal{F}}
\newcommand{\sG}{\mathcal{G}}
\newcommand{\sH}{\mathcal{H}}
\newcommand{\sK}{\mathcal{K}}
\newcommand{\sO}{\mathcal{O}}

\newcommand{\sX}{\mathcal{X}}

\newcommand{\F}{\mathbb{F}}
\newcommand{\G}{\mathbb{G}}
\renewcommand{\H}{\mathbb{H}}
\newcommand{\Q}{\mathbf{Q}}
\newcommand{\Z}{\mathbf{Z}}

\newcommand{\DM}{\operatorname{\bf DM}}
\newcommand{\Sm}{\operatorname{\bf Sm}}
\newcommand{\Hom}{\operatorname{Hom}}
\newcommand{\uHom}{\operatorname{\underline{Hom}}}
\newcommand{\car}{\operatorname{char}}
\newcommand{\Tor}{\operatorname{Tor}}
\newcommand{\Ker}{\operatorname{Ker}}
\newcommand{\Coker}{\operatorname{Coker}}
\newcommand{\Spec}{\operatorname{Spec}}
\newcommand{\Zar}{{\operatorname{Zar}}}
\newcommand{\Nis}{{\operatorname{Nis}}}
\newcommand{\et}{{\operatorname{\acute{e}t}}}
\newcommand{\eff}{{\operatorname{eff}}}

\newcommand{\by}{\xrightarrow}
\newcommand{\yb}{\xleftarrow}
\newcommand{\iso}{\by{\sim}}
\newcommand{\osi}{\yb{\sim}}

\newenvironment{thlist}{\begin{list}{\rm{(\roman{enumi})}}%
{\usecounter{enumi}}}%
{\end{list}}

\swapnumbers
\newtheorem{thm}{Theorem}[section]
\newtheorem{prop}[thm]{Proposition}
\newtheorem{lemma}[thm]{Lemma}
\newtheorem{cor}[thm]{Corollary}
\newtheorem{conj}[thm]{Conjecture}
\theoremstyle{definition}
\newtheorem{defn}[thm]{Definition}

\theoremstyle{remark}

\newtheorem{rks}[thm]{Remarks}

\numberwithin{equation}{section}

\begin{document}

\title{Divisibility properties of motivic cohomology}
\author{Bruno Kahn}
\address{IMJ-PRG\\Case 247\\4 place Jussieu\\75252
Paris Cedex 05\\France}
\email{bruno.kahn@imj-prg.fr}
\date{January 6, 2018}

\begin{abstract}
We extend results of Colliot-Th\'el\`ene and Raskind on the $\sK_2$-cohomology of smooth projective varieties over a separably closed field $k$ to the \'etale motivic cohomology of smooth, not necessarily projective, varieties over $k$. Some consequences are drawn, such as the degeneration of the Bloch-Lichtenbaum spectral sequence for any field containing $k$.
\end{abstract}
\maketitle

\section*{Introduction} In \cite{ctr}, Colliot-Th\'el\`ene and Raskind study the structure of the $\sK_2$-cohomology groups of a smooth projective variety $X$ over a separably closed field. Following arguments of Bloch \cite{bloch}, their proofs use the Weil conjecture proven by Deligne \cite{weilI} and the Merkurjev-Suslin theorem \cite{ms}. These results and proofs can be reformulated in terms of motivic cohomology, since
\[H^i(X,\sK_2)\simeq H^{i+2}(X,\Z(2))\]
or even in terms of \'etale motivic cohomology, since
\[H^j(X,\Z(2))\iso H^j_\et(X,\Z(2))\text{ for }j\le 3\]
as follows again from the Merkurjev-Suslin theorem. If we work in terms of \'etale motivic cohomology, the recourse to the latter theorem is irrelevant and only the results of \cite{weilI} are needed; in this form, the results of \cite{ctr} and their proofs readily extend to \'etale motivic cohomology of higher weights, as in \cite[Prop. 1]{indec} and \cite[Prop. 1.3]{rs}.

In this note, we generalise these results to the \'etale motivic cohomology of smooth varieties over a separably closed field: see Theorem \ref{t2}. This could be reduced by a d\'evissage to the the smooth projective case, using de Jong's alteration theorem in the style of \cite{finiteness}, but it is simpler to reason directly by using cohomology with compact supports, and Weil II \cite{weilII} rather than Weil I \cite{weilI}. I thank H\'el\`ene Esnault and Eckart Viehweg for suggesting to use this approach. This descends somewhat to the case where the base field $k$ is not separably closed, yielding information on the Hochschild-Serre filtration on \'etale motivic cohomology (Theorem \ref{t2.1}). The rest of the note is concerned with implications on motivic cohomology of a field $K$ containing a separably closed field: the main result, which uses the norm residue isomorphism theorem of Voevodsky, Rost et al \cite{voeann}, is that $H^i(K,\Z(n))$ is divisible for $i\ne n$ (Theorem \ref{t3}). As an immediate consequence, the ``Bloch-Lichtenbaum'' spectral sequence of $K$ from motivic cohomology to algebraic $K$-theory degenerates (Theorem \ref{t4}).  We also show that the cokernel of the cup-product map 
\[H^{i-1}(K,\Z(n-1))\otimes K^*\to H^i(K,\Z(n))\]
is uniquely divisible for $i<n$ (Theorem \ref{t5}).

These results were mainly found in 2009, but I did not intend to make them public because I didn't find them sufficiently interesting. Luca Barbieri-Viale informed me that he could use some of them in forthcoming work with F. Andreatta and F. Bertapelle: I thank him for convincing me to turn them into the present preprint. 

In all this note, motivic cohomology is understood in the sense of Suslin and Voevodsky (hypercohomology of the Suslin-Voevodsly complexes \cite{suvo}).

\section{A weight and coniveau argument}

Let $X$ be a separated scheme of finite type over a finitely generated field $k$. %For $m\in\Z$, we write 
%\[\widehat{\Z}'(m) = \prod_l \Z_l(m)\]
%where $l$ runs through the prime numbers invertible in $k$.

\begin{prop}\label{p2.1} Let $n\in\Z$, $k_s$ a separable closure of $k$ and $G=Gal(k_s/k)$.
Let $\bar X=X\otimes_k k_s$. Then $H^j_c(\bar X,\Z_l(m))^G$ and $H^j_c(\bar X,\Z_l(m))_G$ are finite for $j\notin [2m,m+d]$ and any  prime number $l$ invertible in $k$, where $d=\dim X$.
\end{prop}

\begin{proof} Suppose first that $k=\F_q$ is finite. By \cite[Cor. 5.5.3 p. 394]{sga7}, the
eigenvalues of Frobenius acting on $H^j_c(\bar X,\Q_l)$ are algebraic integers which are
divisible by $q^{j-d}$ if $j\ge d$. This yields the necessary bound $m\ge j-d$ for an
eigenvalue $1$. On the other hand, by \cite{weilII}, these eigenvalues have archimedean absolute
values $\le q^{j/2}$: this gives the necessary bound $m\le j/2$ for an eigenvalue $1$. %Finally, for an eigenvalue $\alpha\ne 1$, $\alpha-1$ is only divisible by finitely many rational primes. 
The conclusion follows.

In general, we may choose a regular model $S$ of $k$, of finite type over $\Spec \Z$, such that
$X$ extends to a compactifiable separated morphism of finite type $f:\sX\to S$. By
\cite[lemma 2.2.2 p. 274 and 2.2.3 p. 277]{sga5}, $R^jf_!\Z_l$ is a constructible $\Z_l$-sheaf
on $S$ and its formation commutes with any base change. Shrinking $S$, we may assume that it is
locally constant and that $l$ is invertible on $S$. For a closed point $s\in S$, this gives an
isomorphism
\[H^j_c(\bar X,\Z_l)\simeq H^j_c(\bar X_s,\Z_l)\]
compatible with Galois action, and the result follows from the first case.
\end{proof}

\begin{cor}\label{c2.1} If $X$ is smooth in Proposition \ref{p2.1}, then $H^i(\bar
X,\Z_l(n))^{(G)}$  is finite for $i\notin [n,2n]$, where the superscript $^{(G)}$ denotes the subset of elements invariant under some open subgroup of $G=Gal(k/k_0)$. If $X$ is smooth projective, then $H^i(\bar X,\Z_l(n))^{(G)}$ is finite for $i\ne 2n$ and $0$ for almost all $l$.
\end{cor}

\begin{proof} By Poincar\'e duality and Proposition \ref{p2.1}, $H^i(\bar
X,\Z_l(n))^{G}$  is finite for $i\notin [n,2n]$; the claim follows since $H^i(\bar X,\Z_l(n)))$ is a finitely generated $\Z_l$-module. In the projective case, the Weil conjecture \cite{weilI} actually gives the finiteness of $H^i(\bar X,\Z_l(n))^{G}$, hence of $H^i(\bar X,\Z_l(n))^{(G)}$, for all $i\ne 2n$. But Gabber's theorem \cite{gabber} says that $H^i(\bar X,\Z_l(n))$ is torsion-free for almost all $l$, hence the conclusion.
\end{proof}

\begin{thm}\label{t2} Let $X$ be a smooth variety over a separably closed field $k$ of exponential characteristic $p$. Then, for $i\notin [n,2n]$, the group $H^i_\et(X,\Z(n))[1/p]$
is an extension of a direct sum $T$ of finite $l$-groups by a divisible group. If $X$ is projective, this is true for all $i\ne 2n$, and $T$ is finite. If $p>1$, $H^i_\et(X,\Z(n))$ is uniquely $p$-divisible for $i<n$. In particular, $H^i_\et(X,\Z(n))\otimes \Q/\Z=0$ for $i<n$.  For $i\le 1$, $H^i_\et(X,\Z(n))$ is
divisible. The sequence
\[0\to H^{i-1}_\et(X,\Q/\Z(n))\to H^{i}_\et(X,\Z(n))\to
H^{i}_\et(X,\Z(n))\otimes \Q\to 0\]
is exact for $i<n$.
\end{thm}

\begin{proof} %This is a well-known argument, cf. \cite[\S\S 1,2]{ctr}. Observe that $X$ is defined over a finitely generated subfield $k_0$ with separable closure $k$. The short exact sequences
%\[0\to H^i_\et(X,\Z(n))/m\to H^i_\et(X,\Z/m(n))\to {}_mH^{i+1}_\et(X,\Z(n))\to 0\]
%for $m$ invertible in $k$ yield (since the middle term is finite) a short exact sequence
%\[0\to \widehat{H^i_\et}'(X,\Z(n)) \to H^i_\et(X,\widehat{\Z}'(n))\to \widehat{T}'(H^{i+1}_\et(X,\Z(n)))\to 0\]
%where $\widehat{\Z}'(m) = \prod_{l\ne \car k} \Z_l(m)$, the left hand group is profinite completion away from the characteristic and the right hand one is the corresponding complete Tate module. In this way we cheaply get a map $H^i_\et(X,\Z(n))\otimes \widehat{\Z}'\to H^i_\et(X,\widehat{\Z}'(n))$. But since \'etale motivic cohomology commutes with filtering inverse limits, this map factors through $H^i_\et(X,\widehat{\Z}'(n))^{(G)}$. By Corollary \ref{c2.1}, this subgroup is a direct sum of finite $l$-groups. The rest of the proof is formal, cf
Away from $p$, it is identical to \cite[proof of Prop. 1]{indec} (which is the projective case) in view of Corollary \ref{c2.1}. The unique $p$-divisibility of $H^i_\et(X,\Z(n))$ for $i<n$ follows from \cite[Th. 8.4]{gl2} and requires no hypothesis on $k$.
%
%To get the last statement, it suffices to observe that $H^i_\et(X,\Z_l)$ is torsion-free for $i\le 1$, as follows for $i=1$ from the surjection $H^0_\et(X,\Q_l)\surj H^0_\et(X,\Q_l/\Z_l)$.
\end{proof}

\begin{cor}\label{c2.2} Let $K$ be a field containing a separably closed field $k$. Then,  for $i<n$, the sequence
\[0\to H^{i-1}_\et(K,\Q/\Z(n))\to H^{i}_\et(K,\Z(n))\to
H^{i}_\et(K,\Z(n))\otimes \Q\to 0\]
is exact and the left group has no $p$-torsion if $p=\car K$.
\end{cor}

\begin{proof} We may assume $K/k$ finitely generated. By Theorem \ref{t2}, this is true for any smooth model of $K$ over $k$, and we pass to the limit (see \cite[Prop. 2.1 b)]{bbki}, or rather its proof, for the commutation of \'etale motivic cohomology with limits).
\end{proof}

\begin{rks} 1) At least away from $p$, the range of ``bad'' $i$'s in Corollary \ref{c2.1} and Theorem \ref{t2} is $[n,2n]$ in general but shrinks to $2n$ when $X$ is projective. If we remove a smooth closed subset, this range becomes
$[2n-1,2n]$. As the proof of Proposition \ref{p2.1} shows, it depends on the length of the weight filtration on $H^*(\bar X,\Q_l)$. If $X=Y-D$, where $Y$ is smooth projective and $D$ is a simple normal crossing divisor with $r$ irreducible components,  the range is $[2n-r,2n]$. It would be interesting to understand the optimal range in general, purely in terms of the geometry of $X$.\\ 
2) Using Proposition \ref{p2.1} or more precisely its proof, one may recover the $l$-local version of \cite[Th.
3]{finiteness} without a recourse to de Jong's alteration theorem. I don't see how to get the global
finiteness of loc. cit. with the present method, because one does not know whether the torsion of $H^j_c(\bar X,\Z_l)$ vanishes for $l$ large when $X$ is not smooth projective.\\ 
%use the fact that if $\alpha$ is a nontrivial eigenvalue of Frobenius which is an algebraic integer, then $\alpha - 1$ is only divisible by finitely many rational primes.\\
3) Using a cycle class map to Borel-Moore $l$-adic homology, one could use
Proposition \ref{p2.1} to extend Theorem \ref{t2} to higher Chow groups of arbitrary
separated $k$-schemes of finite type. Such a cycle class map was constructed in
\cite[\S 1.3]{glrev1}. Note that Borel-Moore $l$-adic cohomology is dual to $l$-adic cohomology with compact supports, so the bounds for finiteness are obtained from those of Proposition \ref{p2.1} by changing signs.
\end{rks}

\section{Descent}

\begin{thm}\label{t2.1} Let $X$ be a smooth variety over a field $k$; write $k_s$ for a separable closure of $k$, $X_s$ for $X\otimes_k k_s$ and $G$ for $Gal(k_s/k)$. For a complex of sheaves $C$ over $X_\et$, write $F^rH^i_\et(X,C)$ for the filtration on $H^i_\et(X,C)$ induced by the Hochschild-Serre spectral sequence
\[E_2^{r,s}(C)=H^r(G,H^s_\et(X_s,C))\Rightarrow H^{r+s}_\et(X,C).\]
Then, for $i<n$, the homomorphism 
\[F^rH^{i-1}_\et(X,\Q/\Z(n))\to F^rH^i_\et(X,\Z(n))\] 
induced by the Bockstein homomorphism $\beta$ is bijective for $r\ge 3$ and surjective for $r=1,2$.
\end{thm}

\begin{proof} By the functoriality of $E_m^{r,s}(C)$ with respect to morphisms of complexes, we have a morphism of spectral sequences
\[\delta_m^{r,s}:E_m^{r,s-1}(\Q/\Z(n))\to E_m^{r,s}(\Z(n))\]
converging to the Bockstein homomorphisms. By Theorem \ref{t2}, $\delta_2^{r,i-r}$ is bijective for $r\ge 2$ and surjective for $r=1$. It follows that, for $m\ge 3$, $\delta_m^{r,i-r}$ is bijective for $r\ge 3$ and surjective for $r=1,2$. The conclusion follows.
\end{proof}

\begin{rks} 1) Of course, $F^rH^i_\et(X,\Z(n))$ is torsion for $r>0$ by a transfer argument, hence is contained in $\beta  H^{i-1}(X,\Q/\Z(n))$. The information of Theorem \ref{t2.1} is that it equals $\beta  F^rH^{i-1}(X,\Q/\Z(n))$.\\
2) For $i\ge n$, we have a similar conclusion for higher values of $r$, with the same proof: this is left to the reader.
\end{rks}

\section{Getting the norm residue isomorphism theorem into play}

Recall that for any field $K$ and any $i\le n$, we have an isomorphism

\begin{equation}\label{eq2.1}
H^{i}(K,\Z(n))\iso H^{i}_\et(K,\Z(n)).
\end{equation}

Indeed, this is seen after localising at $l$ for all prime numbers $l$. For $l\ne \car K$, this follows from \cite{suvo,gl} and the norm residue isomorphism theorem \cite{voeann}, while for $l=\car K$ it follows from  \cite{gl2}. Finally, $H^{i}(K,\Z(n))=0$ for $i>n$. This yields:

\begin{thm}\label{t3} Let $K$ be as in Corollary \ref{c2.2}. %Suppose that the Bloch-Kato conjecture is true for the prime $l$ in degree $i-1$. 
Then, for $i\ne n$, the group in \eqref{eq2.1}
 is divisible. %under the Bloch-Kato conjecture in degree $n$.
\end{thm}

\begin{proof} Again it suffices to prove this statement after tensoring with $\Z_{(l)}$ for all prime numbers $l$. This is an immediate consequence of Corollary \ref{c2.2} since, by  \cite{voeann}, one has an isomorphism for $l\ne \car k$
\[H^{i-1}_\et(K,\Q_l/\Z_l(n))\simeq K_{i-1}^M(K)\otimes \Q_l/\Z_l(n-i+1)\] 
and the right hand side is divisible.
\end{proof}

\section{Application: degeneration of the Bloch-Lichtenbaum spectral sequence}

\begin{thm}\label{t4} Let $K$ be as in Corollary \ref{c2.2}. Then the Bloch-Lichten\-baum spectral sequence  \cite[(1.8)]{levine1}
\[E_2^{p,q}=H^{p-q}(K,\Z(-q))\Rightarrow K_{-p-q}(K)\]
degenerates. For any $n>0$, the map $K_n^M(K)\to K_n(K)$ is injective with divisible cokernel.
\end{thm}

\begin{proof} By the Adams operations, the differentials are torsion \cite[Th. 11.7]{levine1}. By Theorem \ref{t3}, they
vanish on the divisible groups $E_2^{p,q}$
for
$p< 0$. But $H^i(K,\Z(n))=0$ for $i>n$, so $E_2^{p,q}=0$ for $p>0$.
The last statement follows from the degeneration plus Theorem \ref{t3}.
\end{proof}

\begin{rks} 1) Again by the Adams operations, the filtration on $K_n(K)$ induced by the
Bloch-Lichtenbaum spectral sequence splits after inverting $(n-1)!$ for any field $K$. On the other hand, we
constructed in \cite{klocal} a canonical splitting of the
corresponding spectral sequence with finite coefficients, including the abutment; the hypothesis that $K$ contains a separably closed field is not required there. This implies
in particular that the map
\[K_n^M(K)/l^\nu\to K_n(K)/l^\nu\]
is split injective \cite[Th. 1 (c)]{klocal}, hence bijective if $K$ contains a separably closed subfield by Theorem \ref{t4}.  
Could it be that the mod $l^\nu$ splittings of \cite{klocal} also exist
integrally?\\
2) As in \cite{klocal}, Theorems \ref{t3} and \ref{t4} extend to regular semi-local rings of
geometric origin containing a separably closed field; the point is that, for such rings
$R$, the groups $H^{i-1}_\et(R,\Q_l/\Z_l(n))$ are divisible by the universal exactness of the
Gersten complexes (\cite{grayson}, \cite[Th. 6.2.1]{cthk}).
\end{rks}

\section{The map $H^{i-1}(K,\Z(n-1))\otimes K^*\to H^i(K,\Z(n))$}

\begin{thm}\label{t5} Let $K$ be as in Corollary \ref{c2.2}. %Suppose that the Bloch-Kato conjecture is true for the prime $l$ in degree $i-1$. 
Then, for $i<n$,
\begin{thlist}
\item The cokernel of the cup-product map
\[H^1(K,\Z(n-i+1))\otimes H^{i-1}(K,\Z(i-1))\by{\gamma^{i,n}} H^i(K,\Z(n))\]
is uniquely divisible. 
\item The cokernel of the cup-product map
\[H^{i-1}(K,\Z(n-1))\otimes K^*\by{\delta^{i,n}} H^i(K,\Z(n))\]
is uniquely divisible. 
\end{thlist}
\end{thm}

\begin{proof} By \eqref{eq2.1}, we may use the \'etale version of these groups. 

(i) Since $H^i_\et(K,\Z(n))$ is divisible
by Theorem \ref{t3}, so is $\Coker \gamma^{i,n}$. Let $\nu\ge 1$ and $l$ prime $\ne \car K$. The diagram
\[\begin{CD}
H^1_\et(K,\Z(n-i+1))\otimes H^{i-1}_\et(K,\Z(i-1))@>\gamma^{i,n}>> H^i_\et(K,\Z(n))\\
@A{\beta\otimes 1}AA @A{\beta}AA\\
H^0_\et(K,\Z/l^\nu(n-i+1))\otimes H^{i-1}_\et(K,\Z(i-1))@>\cup>> H^{i-1}_\et(K,\Z/l^\nu(n))
\end{CD}\]
commutes, where $\beta$ denotes Bockstein. The bottom horizontal map is surjective (even bijective) by the
%Bloch-Kato conjecture. 
norm residue isomorphism theorem (resp. by \cite{gl2}). By Theorem \ref{t3} again, $H^1_\et(K,\Z(n-i+1))$ is $l$-divisible, hence so
is $H^1_\et(K,\Z(n-i+1))\otimes H^{i-1}_\et(K,\Z(i-1))$, and $\Coker \gamma^{i,n}$ is also
$l$-torsion free by an easy diagram chase.

(ii) Consider the commutative diagram
\[\begin{CD}
H^1_\et(K,\Z(n-i+1))\otimes H^{i-2}_\et(K,\Z(i-2))\otimes K^*@>{\gamma^{i-1,n-1}\otimes 1}>> H^{i-1}_\et(K,\Z(n-1))\otimes K^*\\
@V{1\otimes \cup}VV @V{\delta^{i,n}}VV\\
H^1_\et(K,\Z(n-i+1))\otimes H^{i-1}_\et(K,\Z(i-1))@>\gamma^{i,n}>> H^i_\et(K,\Z(n)).
\end{CD}\]

Since the left vertical map is surjective, we see that $\Coker \delta^{i,n}$ is the quotient of $\Coker \gamma^{i,n}$ by the image of the divisible group $H^{i-1}_\et(K,\Z(n-1))\otimes K^*$ (Theorem \ref{t3}), hence the claim follows from (i).
\end{proof}

It would be very interesting to describe $\Ker \delta^{i,n}$, but this seems out
of range.

\end{document}